\documentclass[]{siamart250211}

\usepackage{graphicx} 
\usepackage[utf8]{inputenc}
\usepackage{makecell}
\usepackage{algorithm,algorithmic} 
\usepackage{multirow}
\usepackage{amsmath,amssymb}
\usepackage{mathtools}
\usepackage{array}
\usepackage{booktabs}
\usepackage{enumitem}
\usepackage[caption=false]{subfig} 
\usepackage{upgreek}
\usepackage{bm}
\usepackage{relsize}
\usepackage[capitalize]{cleveref}
\usepackage{xcolor}
\usepackage[normalem]{ulem}
\usepackage{caption}

\usepackage{listings}
\definecolor{codegreen}{rgb}{0,0.6,0}
\definecolor{codegray}{rgb}{0.5,0.5,0.5}
\definecolor{codepurple}{rgb}{0.58,0,0.82}
\definecolor{backcolour}{rgb}{0.95,0.95,0.92}

\crefname{lstlisting}{listing}{listings}
\Crefname{lstlisting}{Listing}{Listings}

\lstdefinestyle{mystyle}{
    backgroundcolor=\color{backcolour},   
    commentstyle=\color{codegreen},
    keywordstyle=\color{magenta},
    numberstyle=\tiny\color{codegray},
    stringstyle=\color{codepurple},
    basicstyle=\ttfamily\footnotesize,
    breakatwhitespace=false,         
    breaklines=true,                 
    captionpos=b,                    
    keepspaces=true,                 
    numbers=left,                    
    numbersep=5pt,                  
    showspaces=false,                
    showstringspaces=false,
    showtabs=false,                  
    tabsize=2
}

\newcommand{\tr}{\operatorname{tr}}

\newcommand{\mat}{\mathbf}

\newcommand{\omg}{\bm{\omega}}
\newcommand{\bmu}{\bm{\mu}}
\newcommand{\bnu}{\bm{\nu}}

\newcommand{\qsubu}{\mat{Q}_{(\mat{u})}}

\newcommand{\LOO}{\texttt{\textup{LeaveOneOut}}}
\newcommand{\XTRACE}{\texttt{\textup{XTrace}}}
\newcommand{\XTF}{\texttt{\textup{XTraceFull}}}
\newcommand{\LOOF}{\texttt{\textup{LeaveOneOutFull}}}

\DeclareMathOperator{\rank}{rank}
\DeclareMathOperator{\orth}{orth}

\DeclareMathOperator{\est}{est}

\DeclareMathOperator{\var}{Var}

\title{Two Variations on the XTrace Algorithm}
\author{Eric Hallman\thanks{Contact: \href{eric.r.hallman@gmail.com}{\texttt{eric.r.hallman@gmail.com}}}}
\date{\today}

\begin{document}

\maketitle
    \begin{changemargin}{1cm}{1cm}
    \begingroup
    \footnotesize
    {\bf Abstract}: This paper studies two potential modifications of \XTRACE{} (Epperly et al., SIMAX 45(1):1-23, 2024), a randomized algorithm for estimating the trace of a matrix. The first is a variance reduction step that averages the output of \XTRACE{} over right-multiplications of the test vectors by random orthogonal matrices. The second is to form a low-rank approximation to the matrix using the whole Krylov space produced by the test vectors, rather than the output of a single power iteration as is used by \XTRACE{}. Experiments on synthetic data show that the first modification offers only slight benefits in practice, while the second can lead to significant improvements depending on the spectrum of the matrix.
    \endgroup
    \end{changemargin}

\section{Introduction}
We'd like to estimate the trace of a matrix $\mat{A}\in\mathbb{R}^{N\times N}$ which can be accessed only via matrix-vector products (matvecs) $\omg\mapsto \mat{A}\omg$. One simple method is to use the Girard-Hutchinson estimator \cite{girard1987algorithme}: if $\omg$ satisfies $\mathbb{E}[\omg\omg^T] = \mat{I}_N$, then 
    \begin{equation*}
        \mathbb{E}[\omg^T\mat{A}\omg] = \mathbb{E}[\tr (\mat{A}\omg\omg^T)] = \tr (\mat{A}\cdot \mathbb{E}[\omg\omg^T]) = \tr(\mat{A}).
    \end{equation*} 
More sophisticated methods reduce the variance of the estimator by constructing a low-rank approximation $\mat{A}\approx \mat{Q}\mat{Q}^T\mat{A}$ and using the identity
    \begin{equation*}
        \tr(\mat{A}) = \tr (\mat{Q}^T\mat{A}\mat{Q}) + \mathbb{E}[\omg^T(\mat{I}-\mat{Q}\mat{Q}^T)\mat{A}(\mat{I}-\mat{Q}\mat{Q}^T)\omg].
    \end{equation*}
For a little extra variance reduction, Epperly et al.~\cite{epperly2024xtrace} recommend a normalization step: first, draw the vectors from a spherically symmetric distribution such as $\omg \sim \mathcal{N}(\mat{0},\mat{I}_N)$. Then, rescale the vectors $(\mat{I}-\mat{Q}\mat{Q}^T)\omg$ to have norm $\sqrt{N-\rank(\mat{Q})}$.  

Meyer et al.~\cite{meyer2021hutch} showed that when $\mat{A}$ is positive definite, it suffices to compute $\mat{Q}$ from a single power iteration $\mat{Q}=\orth(\mat{A}\mat{\Omega})$. However, to na\"{i}vely use the same set of test vectors $\mat{\Omega}$ for both the low-rank approximation $\mat{Q}$ and the Girard-Hutchinson estimator would produce a biased estimate. One solution is to use different sets of test vectors $\mat{\Omega}_1$ and $\mat{\Omega}_2$ for each step: in \cite[Alg.~1]{meyer2021hutch} $\mat{\Omega}_1$ and $\mat{\Omega}_2$ are set to have the same number of columns, and Persson et al.~\cite{persson2022improved} propose an adaptive method to more efficiently allocate the test vectors.

Epperly et al.~\cite{epperly2024xtrace} made the key insight that no such compromise is needed. Their \XTRACE{} estimator uses all of the test vectors for both low-rank approximation and the Girard-Hutchinson estimator, using a leave-one-out technique to ensure that the estimator remains unbiased. Algorithm \ref{alg:xtrace_naive} provides a simple implementation; the authors also show how to compute each of the bases $\mat{Q}_{(i)}$ efficiently so that, for $m$ test vectors, the overall method requires only $\mathcal{O}(m^2N)$ operations besides the matvecs $\mat{\Omega}\mapsto \mat{A}\mat{\Omega}$. 

\begin{algorithm}[ht]
    \caption{XTrace: na\"{i}ve implementation (with normalization) \cite{epperly2024xtrace}}\label{alg:xtrace_naive}
    \fontsize{10}{14}\selectfont
\begin{algorithmic}[1]
\REQUIRE{Matrix $\mat{A}\in \mathbb{R}^{N\times N}$; number of test vectors $m$}
\ENSURE{Approximation to $\tr(\mat{A})$}
\STATE{Sample i.i.d.~$\omg_1,\ldots,\omg_{m} \sim \mathcal{N}(\mat{0},\mat{I}_N)$}
\STATE{$\mat{Y} = \mat{A}\mat{\Omega}$}\hfill\COMMENT{$\mat{\Omega}=[\omg_1,\ldots,\omg_{m}]$}
\FOR{$i=1,2,\ldots,m$}
\STATE{Compute an orthonormal basis $\mat{Q}_{(i)}$ for $\mat{Y}_{-i}$}\hfill\COMMENT{Remove $i$th column of $\mat{Y}$}
\STATE{$\bmu = (\mat{I}-\mat{Q}_{(i)}\mat{Q}_{(i)}^T)\omg_i$}
\STATE{$\bnu = \sqrt{N-\rank(\mat{Q}_{(i)})}\cdot\bmu/\|\bmu\|_2$}\hfill\COMMENT{Normalization}
\STATE{$\widehat{\tr}_i = \tr(\mat{Q}_{(i)}^T\mat{A}\mat{Q}_{(i)}) + \bnu^T\mat{A}\bnu$}
\ENDFOR
\RETURN{$\frac{1}{m}\sum_{i=1}^{m}\widehat{\tr}_i$}
\end{algorithmic}
\end{algorithm}

\phantomsection
\subsection{Contributions}\label{sec:contributions}
This paper proposes two modifications to the \XTRACE{} estimator that will further reduce its variance without adding bias.
\begin{enumerate}
    \item Average the output of \XTRACE{} over multiple random rotations $\mat{\Omega}\mapsto \mat{\Omega}\mat{U}$.
    \item Orthogonalize each $\omg_i$ against $[(\mat{A}\mat{\Omega})_{-i},\, \mat{\Omega}_{-i}]$, not just $(\mat{A}\mat{\Omega})_{-i}$.
\end{enumerate}
Neither of these proposals requires any additional matvecs with $\mat{A}$. 

We also consider the more general question: what invariance properties must an optimal trace estimator have? For the purposes of this paper, a trace estimator $\est_{\mat{A}}(\cdot)$ could be called ``optimal'' if it is admissible within the class
\begin{equation} \label{def:class_for_admissibility}
    \mathcal{F} := \{g(\mat{\Omega}, \mat{A}\mat{\Omega},\mat{A}^2\mat{\Omega})\ :\ \mathbb{E}[g(\mat{\Omega}, \mat{A}\mat{\Omega},\mat{A}^2\mat{\Omega})] = \tr(\mat{A})\}.
\end{equation}
That is, $\est_{\mat{A}}(\cdot)\in \mathcal{F}$, and there is no other estimator $\est_{\mat{A}}'(\cdot)$ satisfying $\var(\est_{\mat{A}}'(\mat{\Omega})) \leq \var(\est_{\mat{A}}(\mat{\Omega}))$ for all $\mat{A}$, with strict inequality for at least one $\mat{A}$. 


We observe that such an estimator must be {\em basis-invariant} in that it can be expressed as a function of $\mathcal{R}(\mat{\Omega})$, and we show in \cref{thm:xtrace_full_basis_invariance} that our modified algorithm has certain invariance properties that \XTRACE{} does not.

\begin{algorithm}[ht]
    \caption{XTrace-Full: na\"{i}ve implementation}\label{alg:xtrace_full_naive}
    \fontsize{10}{14}\selectfont
\begin{algorithmic}[1]
\REQUIRE{Matrix $\mat{A}\in \mathbb{R}^{N\times N}$; number of test vectors $m$}
\ENSURE{Approximation to $\tr(\mat{A})$}
\STATE{Sample i.i.d.~$\omg_1,\ldots,\omg_{m} \sim \mathcal{N}(\mat{0},\mat{I}_N)$}
\STATE{$\mat{Y} = \mat{A}\mat{\Omega}$}\hfill\COMMENT{$\mat{\Omega}=[\omg_1,\ldots,\omg_{m}]$}
\FOR{$i=1,2,\ldots,m$}
\STATE{Compute an orthonormal basis $\mat{Q}_{(i)}$ for $[\mat{Y}_{-i},\,\mat{\Omega}_{-i}]$}
\STATE{$\bmu = (\mat{I}-\mat{Q}_{(i)}\mat{Q}_{(i)}^T)\omg_i$}
\STATE{$\bnu = \sqrt{N-\rank(\mat{Q}_{(i)})}\cdot\bmu/\|\bmu\|_2$}\hfill\COMMENT{Normalization}
\STATE{$\widehat{\tr}_i = \tr(\mat{Q}_{(i)}^T\mat{A}\mat{Q}_{(i)}) + \bnu^T\mat{A}\bnu$}
\ENDFOR
\RETURN{$\frac{1}{m}\sum_{i=1}^{m}\widehat{\tr}_i$}
\end{algorithmic}
\end{algorithm}

\phantomsection
\subsection{Related work and commentary on optimality conditions}
It is worth mentioning that a variety of trace estimators exist that are not in the class of functions defined in \cref{def:class_for_admissibility}.

First, we could permit biased estimators, comparing them by mean-squared error instead of variance. One such estimator follows from the Variance-reduced Stochastic Lanczos Quadrature proposed by Bhattacharjee et al.~\cite[Thm.~5]{bhattacharjee2025improved}, which constructs a Krylov space and deflates sufficiently-converged eigenvalues, resulting in a very small but nonzero bias. More generally, when computing the trace of some matrix function such as $f(\mat{A}) = \mat{A}^{-1}$ or $\exp(\mat{A})$, it is common to approximate the function using Chebyshev polynomials or quadrature. Truly unbiased estimators exist \cite{rhee2015unbiased}, but have not gotten much attention in the literature on trace estimation. 

Second, the requirement that $\mat{A}$ be accessible strictly through matvecs precludes probing methods, which in addition to applying matvecs will partition sparse matrices using graph-coloring algorithms (see \cite{frommer2021analysis} and the references therein). It also precludes multilevel methods such as in Frommer et al.~\cite{frommer2022multigrid}, which are effective when $f(\mat{A})$ can be approximated using a multigrid hierarchy. 

Third, we could consider adaptive methods for choosing how to apply the matvecs: given a budget on the number of matvecs, matrix loads, and memory, one could define
\[\est_{\mat{A}}(\mat{\Omega}) = g(\mat{Y}_0, \mat{Y}_1,\ldots, \mat{Y}_{q-1}),\]
where $\mat{Y}_0=\mat{\Omega}$, $\mat{Y}_1=\mat{A}\cdot g_1(\mat{Y}_0)$, $\mat{Y}_2 = \mat{A}\cdot g_2(\mat{Y}_0,\mat{Y}_1)$, and so on. This function class would be broad enough to accommodate the adaptive method proposed in \cite{persson2022improved}, or the multi-polynomial method in \cite{hallman2022multilevel} which approximates $f(\mat{A})$ using polynomials of varying degree. 

\phantomsection
\subsection{XTrace is not rotation-invariant}
The \XTRACE{} estimator is symmetric with respect to the test vectors $\omg_1,\ldots,\omg_m$. Epperly et al., citing Halmos \cite{halmos1946theory}, note that any minimum-variance unbiased trace estimator must have this property. However, \XTRACE{} is not invariant under arbitrary rotations $\mat{\Omega}\mapsto \mat{\Omega}\mat{U}$. For a reproducible example, consider the inputs
    \[
    \mat{A}=\begin{bmatrix}
        5 & 0 & 0 & 0 & 0 \\
        0 & 4 & 0 & 0 & 0 \\
        0 & 0 & 3 & 0 & 0 \\
        0 & 0 & 0 & 2 & 0 \\
        0 & 0 & 0 & 0 & 1 \\
    \end{bmatrix},\quad
    \mat{\Omega}=\begin{bmatrix}
        1 & 0 \\
        0 & 1/2 \\
        0 & 1/2 \\
        0 & 1/2 \\
        0 & 1/2 \\
    \end{bmatrix}, \quad \mat{U}_\theta=\begin{bmatrix}
        \cos\theta & -\sin\theta \\
        \sin\theta & \cos\theta
    \end{bmatrix}.
    \]
    \cref{fig:xtrace_not_invariant} shows the output of \cref{alg:xtrace_naive} when the random samples are replaced by $\mat{\Omega}\mat{U}_\theta$. As $\theta$ ranges over $[0,\pi/2]$, the trace estimate oscillates. Thus by taking the average output over $\theta$, we could obtain an estimator with lower variance.
\begin{figure}
    \centering
    \includegraphics[width=0.6\linewidth]{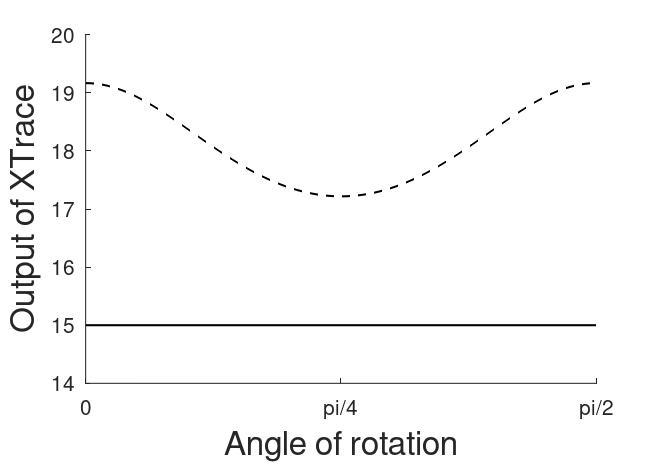}
    \caption{The output of \XTRACE{} is not invariant under arbitrary rotations $\mat{\Omega}\mapsto \mat{\Omega}\mat{U}$.}
    \label{fig:xtrace_not_invariant}
\end{figure}

\phantomsection
\subsection{Organization}
The rest of this paper is organized as follows: \cref{sec:basis_invariant} discusses the invariance properties of various trace estimators and proposes our modified estimator \XTF{} (\cref{alg:xtrace_full_naive}). \Cref{sec:practical_implementation} derives an efficient implementation (\cref{alg:xtrace_full_sampling}). \Cref{sec:numerical_experiments} presents the results of some numerical experiments. \Cref{sec:concluding_remarks} offers concluding remarks and some avenues for future work.

\section{Notation}
For a matrix $\mat{X}$, $\mat{X}_{-i}$ denotes $\mat{X}$ but with the $i$th column removed and $\mathcal{R}(\mat{X})$ denotes the column space of $\mat{X}$. The vector $\mat{e}_i$ is the $i$th standard basis vector. $\mat{X}^{\dagger}$ is the Moore-Penrose inverse, equal to $(\mat{X}^T\mat{X})^{-1}\mat{X}^T$ when $\mat{X}$ has full column rank. $\mat{X}\otimes \mat{Y}$ is the Kronecker product of $\mat{X}$ and $\mat{Y}$, having $(i,j)$ block entry equal to $\mat{X}_{ij}\mat{Y}$. 

A ``random orthogonal'' matrix always means one sampled from the Haar distribution. $\mathcal{N}(\mat{0},\mat{I}_N)$ is the standard multivariate Gaussian distribution, and $\chi_{N}^2$ is the chi-squared distribution with $N$ degrees of freedom.

\section{Basis-invariant estimators} \label{sec:basis_invariant}
Let $\est_{\mat{A}}(\mat{\Omega})$ be a trace estimator with finite variance. The law of total variance states that
\begin{equation*}
    \var(\est_{\mat{A}}(\mat{\Omega})) =
    \mathbb{E}[\var(\est_{\mat{A}}(\mat{\Omega})|\mathcal{R}(\mat{\Omega}))]
    + \var\left(\mathbb{E}[\est_{\mat{A}}(\mat{\Omega})|\mathcal{R}(\mat{\Omega})]\right),
\end{equation*}
and the law of total expectation states that
\begin{equation*}
    \mathbb{E}[\mathbb{E}[\est_{\mat{A}}(\mat{\Omega})|\mathcal{R}(\mat{\Omega})]] = \mathbb{E}[\est_{\mat{A}}(\mat{\Omega})].
\end{equation*}
Therefore, by conditioning on $\mathcal{R}(\mat{\Omega})$ and taking expectations, we can obtain an estimator with the same expectation and lesser-or-equal variance. Such an estimator is also {\em basis-invariant}: it depends only on $\mathcal{R}(\mat{\Omega})$, and transformations $\mat{\Omega}\mapsto \mat{\Omega}\mat{B}$ will not affect the output for any invertible matrix $\mat{B}$. 

When $\mat{\Omega}$ has i.i.d.~standard Gaussian entries, we can use a standard result about the distributions of such matrices \cite{Muirhead1982AspectsOM} to express the conditional expectation in a cleaner form.

\begin{lemma}\label{lemma:gaussian_qr}
    If $\mat{\Omega}\in \mathbb{R}^{N\times m}$ has i.i.d.~$\mathcal{N}(0,1)$ entries, then it has distribution equal to a product of three independent terms $\mat{Q}\mat{U}\mat{R}$, where 
    \begin{itemize}
        \item $\mat{Q}\in \mathbb{R}^{N\times m}$ has orthonormal columns and is Haar-distributed on the Stiefel manifold $V_m(\mathbb{R}^N)$;
        \item $\mat{U}\in \mathbb{R}^{m\times m}$ is orthogonal and Haar-distributed;
        \item $\mat{R}$ is upper triangular with independent entries $\mat{R}_{ii}\sim \chi_{N-i+1}^2$, $\mat{R}_{ij}\sim \mathcal{N}(0,1)$, and consequently $\mathbb{E}[\mat{R}^T\mat{R}]=N\cdot \mat{I}_m$.
    \end{itemize}
    Furthermore, conditioning on $\mathcal{R}(\mat{\Omega})$ is equivalent to conditioning on $\mat{Q}$. 
\end{lemma}

Thus in principle, we could always estimate $\mathbb{E}[\est_{\mat{A}}(\mat{\Omega})|\mathcal{R}(\mat{\Omega})]$ by fixing an orthonormal basis $\mat{Q}$ for $\mathcal{R}(\mat{\Omega})$ and repeatedly sampling $\est_{\mat{A}}(\mat{Q}\mat{U}\mat{R})$ for random $\mat{U}$ and $\mat{R}$. Depending on the estimator, it may be possible to simplify the expression further. In the case of the plain Girard-Hutchinson estimator, for example, we can compute the conditional expectation exactly.

\begin{theorem}\label{thm:hutch_basis_invariance}
    If $\mat{\Omega}\in \mathbb{R}^{N\times m}$ has i.i.d.~$\mathcal{N}(0,1)$ entries and $\mat{Q}$ is any orthonormal basis for $\mathcal{R}(\mat{\Omega})$, then  
    \begin{equation*}
        \mathbb{E}\left[\tfrac{1}{m}\tr (\mat{\Omega}^T\mat{A}\mat{\Omega})|\mathcal{R}(\mat{Q})\right] = \tfrac{N}{m}\tr(\mat{Q}^T\mat{A}\mat{Q}).
    \end{equation*}
\end{theorem}
\begin{proof}
    By \cref{lemma:gaussian_qr} and the law of total expectation, we find that
    \begin{align*}
    \mathbb{E}\left[\tfrac{1}{m}\tr (\mat{\Omega}^T\mat{A}\mat{\Omega})|\mathcal{R}(\mat{\Omega})\right]
        &= \tfrac{1}{m}\mathbb{E}[\tr (\mat{R}^T\mat{U}^T\mat{Q}^T\mat{A}\mat{Q}\mat{U}\mat{R})|\mat{Q}] \\
        &= \tfrac{1}{m}\tr(\mat{Q}^T\mat{A}\mat{Q}\cdot\mathbb{E}[\mat{U}\mat{R}\mat{R}^T\mat{U}^T])\\
        &= \tfrac{1}{m}\tr(\mat{Q}^T\mat{A}\mat{Q}\cdot\mathbb{E}[\mathbb{E}[\mat{U}\mat{R}\mat{R}^T\mat{U}^T|\mat{R}]])\\
        &= \tfrac{1}{m}\tr\left(\mat{Q}^T\mat{A}\mat{Q}\cdot \mathbb{E}\left[\tfrac{1}{m}\tr(\mat{R}\mat{R}^T)\mat{I}_m\right]\right)\\
        &= \tfrac{1}{m^2}\tr\left(\mat{Q}^T\mat{A}\mat{Q} \right)\cdot \mathbb{E}\left[\tr(\mat{R}^T\mat{R})\right]\\
        &= \tfrac{1}{m^2}\tr\left(\mat{Q}^T\mat{A}\mat{Q}\right)\cdot\mathbb{E}\left[\|\mat{\Omega}\|_F^2\right]\\
        &= \tfrac{1}{m^2}\tr\left(\mat{Q}^T\mat{A}\mat{Q}\right)\cdot mN\\
        &= \tfrac{N}{m}\tr(\mat{Q}^T\mat{A}\mat{Q}).
\end{align*}
In the above, we used the fact that $\mathbb{E}[\mat{U}\mat{M}\mat{U}^T] = \tfrac{1}{m}\tr(\mat{M})\mat{I}_m$ for symmetric $\mat{M}$, and the linear and cyclic properties of the trace operator.
\end{proof}
It is already generally known that first orthogonalizing $\mat{\Omega}$ slightly reduces the variance of the Girard-Hutchinson estimator; this proof frames that result in terms of conditional expectations.

\phantomsection
\subsection{Application to XTrace}
For \XTRACE{}, on the other hand, some numerical experiments suggest that 
\begin{equation}
    \mathbb{E}[\XTRACE_{\mat{A}}(\mat{Q}\mat{U}\mat{R})|\mat{Q})] \neq \mathbb{E}[\XTRACE_{\mat{A}}(\mat{Q}\mat{U})|\mat{Q})]
\end{equation}
in general. If so, then there is no result for \XTRACE{} exactly analogous to \cref{thm:hutch_basis_invariance} which by taking expectations removes the dependence on $\mat{R}$ entirely. 

There is, however, an analogous result for \XTF{}. To derive it, we introduce \cref{alg:leave_one_out_full} as a computational primitive called \LOOF{}. It uses the first $m-1$ columns of $\mat{\Omega}$ for deflation, and reserves the final column for the Girard-Hutchinson estimator. \XTF{} is then the average of \LOOF{} over all column permutations $\mat{\Omega} \mapsto \mat{\Omega}\mat{\Pi}$. Similarly, \cref{alg:leave_one_out} describes a computational primitive \LOO{} for \XTRACE{}.

\begin{lemma}\label{lemma:loo_u_last_column}
    For orthogonal $\mat{U}$ and fixed $\mat{\Omega}$, the value of the functions
    \begin{align*}
        \mat{U} &\mapsto \LOOF_{\mat{A}}(\mat{\Omega}\mat{U}),\\
        \mat{U} &\mapsto \LOO_{\mat{A}}(\mat{\Omega}\mat{U})
    \end{align*} depend only on the final column of $\mat{U}$. 
\end{lemma}
\begin{proof}
    Let $\mat{u}$ be the final column of $\mat{U}$. Since
    \begin{align*}
        \mathcal{R}\left((\mat{A}\mat{\Omega}\mat{U})_{-m}\right)
        &= \mathcal{R}\left(\mat{A}\mat{\Omega}\mat{U}(\mat{I}-\mat{e}_m\mat{e}_m^T)\right) \\
        &= \mathcal{R}\left(\mat{A}\mat{\Omega}(\mat{I}-\mat{u}\mat{u}^T)\mat{U}\right)\\
        &= \mathcal{R}\left(\mat{A}\mat{\Omega}(\mat{I}-\mat{u}\mat{u}^T)\right),
    \end{align*}
    and similarly for $\mathcal{R}\left((\mat{\Omega}\mat{U})_{-m}\right)$, the range of the orthonormal basis $\mat{Q}$ computed by \cref{alg:leave_one_out_full} or \cref{alg:leave_one_out} is determined entirely by $\mat{u}$ for any given $\mat{\Omega}$. After that, the algorithms compute
    \begin{equation*}
        \bmu = (\mat{I}-\mat{Q}\mat{Q}^T)\mat{\Omega}\mat{U}\mat{e}_m = (\mat{I}-\mat{Q}\mat{Q}^T)\mat{\Omega}\mat{u},
    \end{equation*}
    which again depends only on $\mat{u}$.
\end{proof}

\begin{lemma}\label{lemma:loo_r_invariant}
    For any invertible block upper triangular matrix $\mat{R}$ with block sizes $(m-1, 1)$, 
    \begin{equation*}
        \LOOF_{\mat{A}}(\mat{\Omega}) = \LOOF_{\mat{A}}(\mat{\Omega}\mat{R}).
    \end{equation*}
    This result does not hold for \LOO{}. 
\end{lemma}
\begin{proof}
    \cref{alg:leave_one_out_full} effectively computes the unpivoted QR factorization 
    \begin{equation*}
        [\mat{A}\mat{\Omega}_{-m},\,\mat{\Omega}] = [\mat{Q},\,\bmu]\widehat{\mat{R}}.
    \end{equation*}
    The transformation $\mat{\Omega}\mapsto \mat{\Omega}\mat{R}$ does not alter either $\mathcal{R}(\mat{Q})$ or $\mathcal{R}(\bmu)$, and so the output of \cref{alg:leave_one_out_full} is unchanged. The same cannot be said of \cref{alg:leave_one_out} because 
    \begin{equation*}
        \mathcal{R}([\mat{A}\mat{\Omega}_{-m},\,\mat{\Omega}\mat{e}_n]) \neq \mathcal{R}([\mat{A}\mat{\Omega}_{-m},\,\mat{\Omega}\mat{r}])
    \end{equation*}
    for general nonzero vectors $\mat{r}$. 
\end{proof}

We are now set up to provide our main result about \XTF{}. 
\begin{theorem}\label{thm:xtrace_full_basis_invariance}
    If $\mat{\Omega}\in \mathbb{R}^{N\times m}$ has i.i.d.~$\mathcal{N}(0,1)$ entries and $\mat{Q}$ is any orthonormal basis for $\mathcal{R}(\mat{\Omega})$, then
    \begin{align*}
        \mathbb{E}[\XTF_{\mat{A}}(\mat{\Omega})|\mathcal{R}(\mat{\Omega})] &= \mathbb{E}[\LOOF_{\mat{A}}(\mat{Q}\mat{U})|\mat{Q}] \\
        &= \mathbb{E}[\XTF_{\mat{A}}(\mat{Q}\mat{U})|\mat{Q}],
    \end{align*}
    for random orthogonal $\mat{U}$. The value of $\LOOF_{\mat{A}}(\mat{Q}\mat{U})$ depends only on $\mat{Q}$ and the last column of $\mat{U}$. 
\end{theorem}
\begin{proof}
    Let $\mat{\Pi}$ be a random permutation. Use, in order, the fact that $\mat{\Omega}$ and $\mat{\Omega}\mat{\Pi}$ have the same distribution conditional on $\mathcal{R}(\mat{\Omega})$, \cref{lemma:gaussian_qr}, and \cref{lemma:loo_r_invariant}:
    \begin{align*}
        \mathbb{E}[\XTF_{\mat{A}}(\mat{\Omega})|\mathcal{R}(\mat{\Omega})] &=
        \mathbb{E}[\mathbb{E}[\LOOF_{\mat{A}}(\mat{\Omega}\mat{\Pi})|\mat{\Omega}]|\mathcal{R}(\mat{\Omega})] \\
        &= \mathbb{E}[\LOOF_{\mat{A}}(\mat{\Omega})|\mathcal{R}(\mat{\Omega})] \\
        &= \mathbb{E}[\LOOF_{\mat{A}}(\mat{Q}\mat{U}\mat{R})|\mat{Q}] \\
        &= \mathbb{E}[\LOOF_{\mat{A}}(\mat{Q}\mat{U})|\mat{Q}].
    \end{align*}
    The second equality follows from the fact that $\mat{U}$ and $\mat{U}\mat{\Pi}$ have the same distribution:
    \begin{align*}
        \mathbb{E}[\XTF_{\mat{A}}(\mat{Q}\mat{U})|\mat{Q}]
        &= \mathbb{E}[\mathbb{E}[\LOOF_{\mat{A}}(\mat{Q}\mat{U}\mat{\Pi})|\mat{Q},\mat{U}]|\mat{Q}] \\
        &= \mathbb{E}[\LOOF_{\mat{A}}(\mat{Q}\mat{U})|\mat{Q}].
    \end{align*}
    The final claim is just restating the result of \cref{lemma:loo_u_last_column}.
\end{proof}
Thus for $\LOOF$, as for the plain Girard-Hutchinson estimator in \cref{thm:hutch_basis_invariance}, the expected output conditional on $\mathcal{R}(\mat{\Omega})$ can be expressed in a form that makes no reference to the triangular factor $\mat{R}$. 

\phantomsection
\subsection{Sampling strategies}\label{sec:sampling_strategies}
To compute $\mathbb{E}[\XTF_{\mat{A}}(\mat{Q}\mat{U})| \mat{Q}]$ or estimate it, there are a few options worth considering:
\begin{itemize}
    \item Express $\mathbb{E}[\LOO_{\mat{A}}(\mat{Q}\mat{U})|\mat{Q}]$ as an integral over the sphere $\mathcal{S}^{m-1}$ and compute it via quadrature. This is feasible if the number of test vectors $m$ is very small (say, $2\leq m \leq 4$).
    \item Approximate $\mathbb{E}[\LOO_{\mat{A}}(\mat{Q}\mat{U})|\mat{Q}]$ by sampling random unit vectors $\mat{u}$. This will come at a marginal cost of $\mathcal{O}(m^2)$ operations per sample.
    \item Approximate $\mathbb{E}[\XTF_{\mat{A}}(\mat{Q}\mat{U})| \mat{Q}]$ by sampling random orthogonal matrices $\mat{U}$. It is plausible that orthonormal vectors will yield negatively correlated estimates, which could make up for the added orthogonalization costs compared to sampling vectors independently.
    \item Approximate $\mathbb{E}[\XTF_{\mat{A}}(\mat{Q}\mat{U})| \mat{Q}]$ by starting with $\mat{U}=\mat{I}_m$ and transforming it using a sequence of random Givens rotations (Kac's random walk) or Householder reflections. 
\end{itemize}

Implemented properly, the computational cost of any of these methods will be independent of the dimension $N$ of $\mat{A}$.

\begin{algorithm}[ht]
\caption{LeaveOneOut}\label{alg:leave_one_out}
    \fontsize{10}{14}\selectfont
\begin{algorithmic}[1]
\REQUIRE{Matrix $\mat{A}\in \mathbb{R}^{N\times N}$; test vectors $\mat{\Omega}\in \mathbb{R}^{N\times m}$}
\ENSURE{Approximation to $\tr(\mat{A})$}
\STATE{Compute an orthonormal basis $\mat{Q}$ for $\mat{A}\mat{\Omega}_{-m}$}\hfill\COMMENT{Remove last column of $\mat{Y}$}
\STATE{$\bmu = (\mat{I}-\mat{Q}\mat{Q}^T)\mat{\Omega}\mat{e}_m$}
\STATE{$\bnu = \sqrt{N-\rank(\mat{Q})}\cdot\bmu/\|\bmu\|_2$}\hfill\COMMENT{Normalization}
\RETURN{$\tr(\mat{Q}^T\mat{A}\mat{Q}) + \bnu^T\mat{A}\bnu$}
\end{algorithmic}
\end{algorithm}

\begin{algorithm}[ht]
\caption{LeaveOneOutFull}\label{alg:leave_one_out_full}
    \fontsize{10}{14}\selectfont
\begin{algorithmic}[1]
\REQUIRE{Matrix $\mat{A}\in \mathbb{R}^{N\times N}$; test vectors $\mat{\Omega}\in \mathbb{R}^{N\times m}$}
\ENSURE{Approximation to $\tr(\mat{A})$}
\STATE{Compute an orthonormal basis $\mat{Q}$ for $[\mat{A}\mat{\Omega}_{-m},\,\mat{\Omega}_{-m}]$}
\STATE{$\bmu = (\mat{I}-\mat{Q}\mat{Q}^T)\mat{\Omega}\mat{e}_m$}
\STATE{$\bnu = \sqrt{N-\rank(\mat{Q})}\cdot\bmu/\|\bmu\|_2$}\hfill\COMMENT{Normalization}
\RETURN{$\tr(\mat{Q}^T\mat{A}\mat{Q}) + \bnu^T\mat{A}\bnu$}
\end{algorithmic}
\end{algorithm}

\section{Practical implementation} \label{sec:practical_implementation}
This section will cover how to compute the value of $\LOO_{\mat{A}}(\mat{Q}\mat{U})$ given the last column $\mat{u}$ of $\mat{U}$. The extension to evaluating $\XTF_{\mat{A}}(\mat{Q}\mat{U})$ given a whole matrix $\mat{U}$ is straightforward. 

First, compute the QR factorization
\begin{equation*}
    \begin{bmatrix}
        \mat{\Omega} & \mat{A}\mat{\Omega}
    \end{bmatrix} = \begin{bmatrix}
        \mat{Q}_0 & \mat{Q}_1
    \end{bmatrix}\begin{bmatrix}
        \mat{R}_0 & \mat{M} \\ & \mat{R}_1
    \end{bmatrix},
\end{equation*}
and right-multiply both block columns by $\mat{R}_0^{-1}$ to get the same factorization over a different set of starting vectors,
\begin{equation} \label{eqn:qr_on_orth_basis}
    \begin{bmatrix}
        \mat{Q}_0 & \mat{A}\mat{Q}_0
    \end{bmatrix} = \underbrace{\begin{bmatrix}
        \mat{Q}_0 & \mat{Q}_1
    \end{bmatrix}}_{\mat{Q}}\underbrace{\begin{bmatrix}
        \mat{I}_m & \mat{M}\mat{R}_0^{-1} \\ & \mat{R}_1\mat{R}_0^{-1}
    \end{bmatrix}}_{\mat{R}}.
\end{equation}
Given a vector $\mat{u}$, we'd like to cheaply find an orthonormal basis for the reduced matrix
\begin{equation*}
    \begin{bmatrix}
        \mat{Q}_0(\mat{I}-\mat{u}\mat{u}^\dagger) & \mat{A}\mat{Q}_0(\mat{I}-\mat{u}\mat{u}^\dagger)
    \end{bmatrix}
    = \mat{Q}\mat{R}\left(\mat{I}_{2m} - \begin{bmatrix}
        \mat{u} & \\
        & \mat{u}
    \end{bmatrix}\begin{bmatrix}
        \mat{u} & \\
        & \mat{u}
    \end{bmatrix}^\dagger\right)
\end{equation*}
The two vectors removed from the basis can be expressed in the form $\mat{Q}\mat{S}$ for some $\mat{S}\in \mathbb{R}^{2m\times 2}$, and must satisfy 
\begin{equation*}
    \mat{Q}\mat{S}\,\bot\, \mat{Q}\mat{R}\left(\mat{I}_{2m} - \begin{bmatrix}
        \mat{u} & \\
        & \mat{u}
    \end{bmatrix}\begin{bmatrix}
        \mat{u} & \\
        & \mat{u}
    \end{bmatrix}^\dagger\right)
\end{equation*}
One solution is given by 
\begin{equation*}
    \mat{S} = \mat{R}^{-T}\begin{bmatrix}
        \mat{u} & \\ & \mat{u}
    \end{bmatrix}.
\end{equation*}
Specifically, the first and second columns are the vectors that would be left out by removing $\mat{Q}_0\mat{u}$ and $\mat{A}\mat{Q}_0\mat{u}$, respectively, from $\begin{bmatrix}
    \mat{Q}_0 & \mat{A}\mat{Q}_0
\end{bmatrix}$.

The result $\mat{S}$ does not in general have orthogonal columns, so find the factorization $\mat{S} = \widetilde{\mat{S}}\mat{L}$, where $\mat{L}$ is {\em lower} triangular. Why lower triangular? Because in order to find $\bmu$ in \cref{alg:leave_one_out_full}, we effectively want the unpivoted QR factorization
\begin{equation*}
    \begin{bmatrix}
        \mat{Q}_0\mat{U}_{\bot} & \mat{A}\mat{Q}_0\mat{U}_\bot & \mat{Q}_0\mat{u} & \mat{A}\mat{Q}_0\mat{u}
    \end{bmatrix}
    =
    \begin{bmatrix}
        \qsubu & \bmu & \mat{Q}\tilde{\mat{s}}_2
    \end{bmatrix},
\end{equation*}
where $[\mat{U}_{\bot},\,\mat{u}]$ is orthogonal. The final column $\mat{Q}\tilde{\mat{s}}_2$ is the result of orthogonalizing $\mat{A}\mat{Q}_0\mat{u}$ against all the other columns, and so will be parallel to $\mat{Q}\mat{s}_2$. But we don't want to orthogonalize $\mat{Q}_0\mat{u}$ against {\em all} other columns, just the ones that form $\mat{Q}_{(\mat{u})}$. So to find $\bmu$ (parallel to $\mat{Q}\tilde{\mat{s}}_1$), we have to orthogonalize $\mat{s}_1$ against $\tilde{\mat{s}}_2$. 

Now to compute the trace estimate. As a preprocessing step, define
\begin{equation}
    \mat{H} := \mat{Q}^T\mat{A}\mat{Q},
\end{equation}
noting that the first $m$ columns of $\mat{H}$ were already computed in \cref{eqn:qr_on_orth_basis}. We then find that
\begin{align*}
    \widehat{\tr}_{\mat{u}} &= \tr(\qsubu^T\mat{A}\qsubu) + \bnu^T\mat{A}\bnu \\
    &= \tr\left(\mat{A}\mat{Q}(\mat{I}-\widetilde{\mat{S}}\widetilde{\mat{S}}^T)\mat{Q}^T\right) + (N - \rank(\mat{Q}))\cdot \tilde{\mat{s}}_1^T\mat{Q}^T\mat{A}\mat{Q}\tilde{\mat{s}}_1 \\
    &= \tr(\mat{H}) - \tr\left(\widetilde{\mat{S}}^T\mat{H}\widetilde{\mat{S}}\right)
    + (N - \rank(\mat{Q}))\cdot \tilde{\mat{s}}_1^T\mat{H}\tilde{\mat{s}}_1 \\
    &= \tr(\mat{H}) - \tilde{\mat{s}}_2^T\mat{H}\tilde{\mat{s}}_2
    + (N - \rank(\mat{Q})- 1)\cdot \tilde{\mat{s}}_1^T\mat{H}\tilde{\mat{s}}_1.
\end{align*}
The end result is presented as \cref{alg:xtrace_full_sampling}, where the outer loop averages the result over multiple random rotations $\mat{\Omega}\mapsto\mat{\Omega}\mat{U}$. As with \XTRACE{}, the loops can be eliminated through vectorization, unless the samples are being computed iteratively via Givens rotations or Householder reflections as mentioned in \cref{sec:sampling_strategies}. In the latter case, only the inner loop can be eliminated.

\begin{algorithm}[ht]
    \caption{XTrace-Full: Efficient implementation}\label{alg:xtrace_full_sampling}
    \fontsize{10}{14}\selectfont
\begin{algorithmic}[1]
\REQUIRE{Matrix $\mat{A}\in \mathbb{R}^{N\times N}$; number of test vectors $m$, resampling number $k$}
\ENSURE{Approximation to $\tr(\mat{A})$}
\STATE{Sample $\mat{\Omega}\in\mathbb{R}^{N\times m}$ with i.i.d.~$\mathcal{N}(0,1)$ entries}
\STATE{Compute an orthonormal basis $\mat{Q}_0$ for $\mat{\Omega}$}
\STATE{QR factorization $[\mat{Q}_0,\,\mat{A}\mat{Q}_0] = \mat{Q}\mat{R}$}\hfill\COMMENT{See \cref{eqn:qr_on_orth_basis}}
\STATE{$\mat{H} = \mat{Q}^T\mat{A}\mat{Q}$}\hfill\COMMENT{Entries overlap with $\mat{R}$}
\FOR{$j=1,2,\ldots,k$}
\STATE{Sample random orthogonal $\mat{U} \in \mathbb{R}^{m\times m}$}\hfill\COMMENT{$\mathcal{O}(m^3)$ operations}
\FOR{$i=1,2,\ldots,m$}
\STATE{$\mat{u} = \mat{U}\mat{e}_i$}
\STATE{Compute $\mat{S} = \mat{R}^{-T}(\mat{I}_2\otimes \mat{u})$}\hfill\COMMENT{$\mathcal{O}(m^2)$ operations}
\STATE{QL factorization $\mat{S} = \widetilde{\mat{S}}\mat{L}$}\hfill\COMMENT{$\mathcal{O}(m)$ operations}
\STATE{$\widehat{\tr}_{ij} = \tr(\mat{H}) - \tilde{\mat{s}}_2^T\mat{H}\tilde{\mat{s}}_2
    + (N - \rank(\mat{Q})- 1)\cdot \tilde{\mat{s}}_1^T\mat{H}\tilde{\mat{s}}_1$}\hfill\COMMENT{$\mathcal{O}(m^2)$ operations}
\ENDFOR
\ENDFOR
\RETURN{$\frac{1}{mk}\sum_{j=1}^{k}\sum_{i=1}^m\widehat{\tr}_{ij}$}
\end{algorithmic}
\end{algorithm}

\section{Numerical experiments} \label{sec:numerical_experiments}
This section presents a comparison of the estimators \XTRACE{} and \XTF{} (both with and without resampling) on a variety of synthetic examples. In all cases we fix the matrix dimension $N=1000$, sample $\mat{\Omega}$ with i.i.d.~$\mathcal{N}(0,1)$ entries, and report the root-mean-squared relative error over 1000 trials. For each trial, each estimator will use the same set of test vectors $\mat{\Omega}$. The variant of \XTF{} with resampling will right-multiply $\mat{\Omega}$ by $25$ randomly sampled orthogonal matrices, the first of which will always be the identity matrix. Code for \XTF{} is given in \cref{listing:xtrace_code}, and the experiments were run using Octave 6.2.0. 

We use diagonal matrices with six choices for the eigenvalues $\bm{\lambda}$:
\begin{itemize}
    \item \texttt{flat}: $\bm{\lambda} = (3-2(i-1)/(N-1):i=1,2,\ldots,N)$;
    \item \texttt{poly}: $\bm{\lambda} = (i^{-2}:i=1,2,\ldots,N)$;
    \item \texttt{inv-poly}: $\bm{\lambda} = (2-i^{-2}:i=1,2,\ldots,N)$;
    \item \texttt{exp}: $\bm{\lambda} = (0.7^i:i=0,1,\ldots,N-1)$;
    \item \texttt{step}: $\bm{\lambda} = (\underbrace{1,\ldots,1}_{\text{50 times}},\underbrace{10^{-3},\ldots,10^{-3}}_{N-50\text{ times}})$;
    \item \texttt{step-decay}: $\bm{\lambda} = (\underbrace{1,\ldots,1}_{\text{50 times}}, 51^{-2},\ldots,N^{-2})$.
\end{itemize}

Results are presented in \cref{fig:synthetic_results}. For four of the six examples, the results are largely the same for all three methods. For the other two, \texttt{inv-poly} and \texttt{step}, \XTF{} performs significantly better than \XTRACE{}. In the first case, it is because the outlying eigenvalues are the ones {\em closer} to zero, which the Krylov space $[\mat{\Omega}, \mat{A}\mat{\Omega}]$ can detect significantly better than the space $\mat{A}\mat{\Omega}$ can. The second case is due to the lagging eigenvalues being precisely equal to one another, in which case a sufficiently large Krylov space captures the leading eigenvalues perfectly. When the lagging values are modified to decay slowly (example \texttt{step-decay}), the difference between the methods vanishes.

\begin{figure}
    \centering
    \begin{minipage}{0.48\textwidth}
        \centering
        \includegraphics[width=\textwidth]{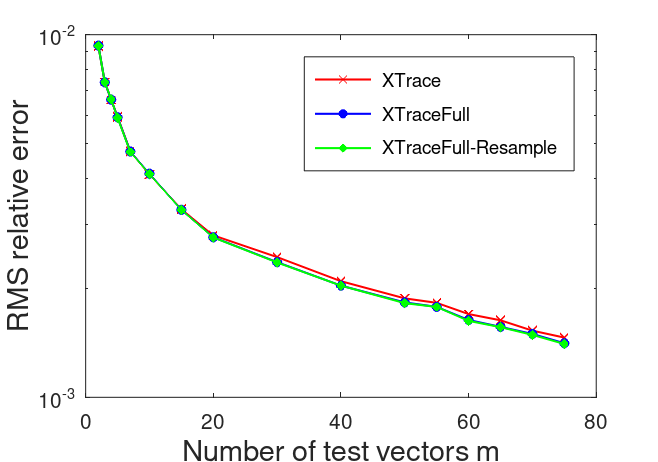}
        \caption*{(a) \texttt{flat}}
    \end{minipage}\hfill
    \begin{minipage}{0.48\textwidth}
        \centering
        \includegraphics[width=\textwidth]{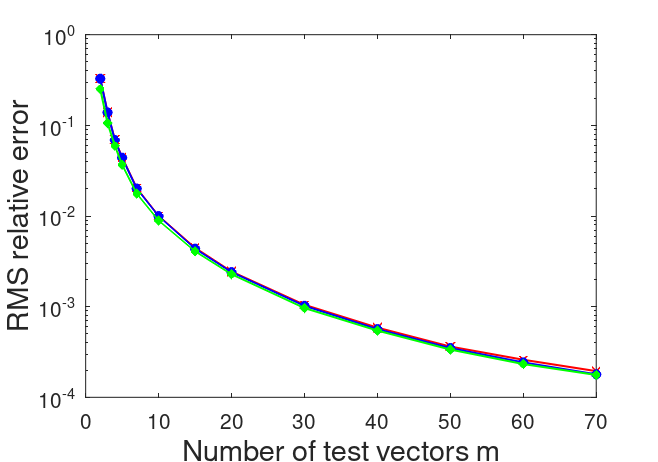}
        \caption*{(a) \texttt{poly}}
    \end{minipage}
    \begin{minipage}{0.48\textwidth}
        \centering
        \includegraphics[width=\textwidth]{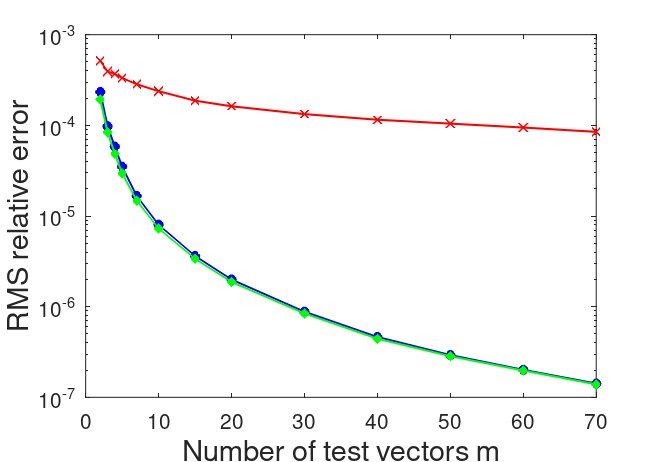} 
        \caption*{(a) \texttt{inv-poly}}
    \end{minipage}\hfill
    \begin{minipage}{0.48\textwidth}
        \centering
        \includegraphics[width=\textwidth]{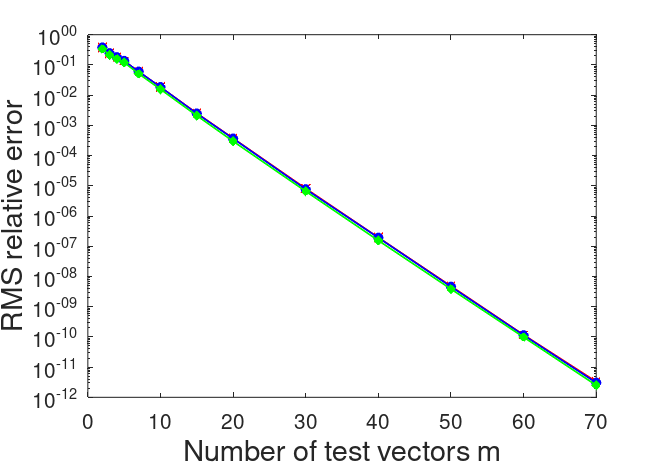}
        \caption*{(a) \texttt{exp}}
    \end{minipage}
    \begin{minipage}{0.48\textwidth}
        \centering
        \includegraphics[width=\textwidth]{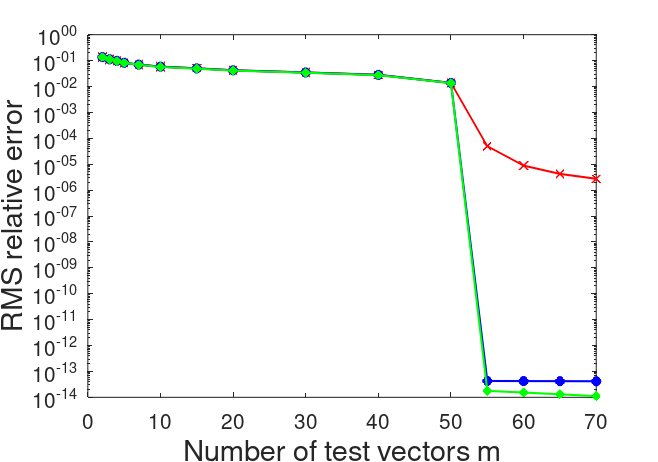} 
        \caption*{(a) \texttt{step}}
    \end{minipage}\hfill
    \begin{minipage}{0.48\textwidth}
        \centering
        \includegraphics[width=\textwidth]{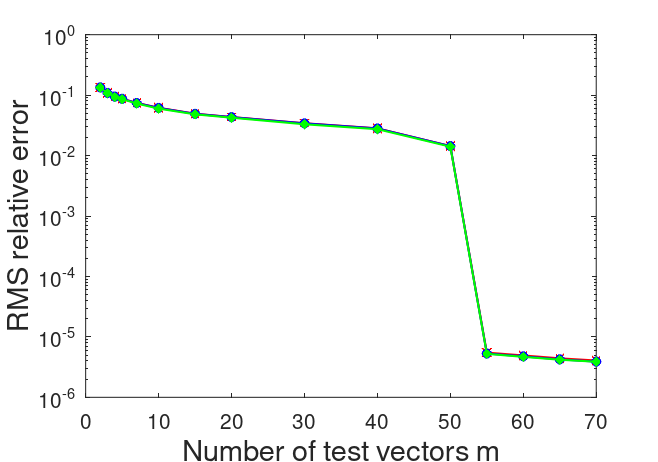}
        \caption*{(a) \texttt{step-decay}}
    \end{minipage}
    \caption{Method \XTF{} performed largely the same as \XTRACE{}, except in cases where the full Krylov space could capture the spectrum of $\mat{A}$ notably better than the power method.}
    \label{fig:synthetic_results}
\end{figure}

\cref{fig:synthetic_ratios} presents a closer view of the effect of resampling $\mat{\Omega}\mat{U}$ for random orthogonal matrices $\mat{U}$. In both cases shown, \texttt{poly} and \texttt{exp}, the benefit of resampling is small but nonzero. For the case \texttt{poly} in particular, the benefit appears when the number of test vectors $m$ is very small, but diminishes as $m$ increases. 

\begin{figure}
    \centering
    \begin{minipage}{0.48\textwidth}
        \centering
        \includegraphics[width=\textwidth]{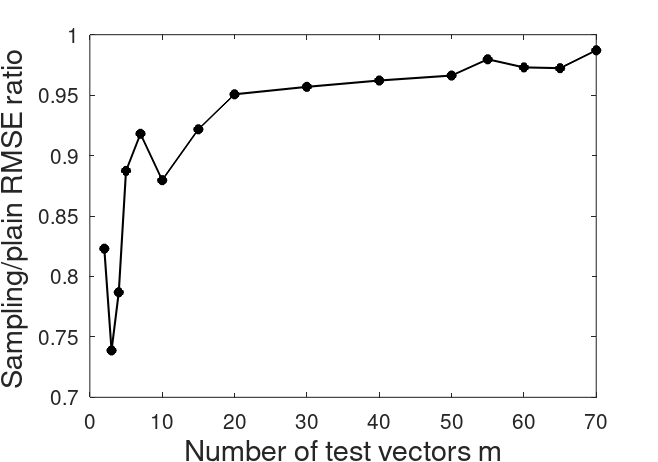}
        \caption*{(a) \texttt{poly}}
    \end{minipage}\hfill
    \begin{minipage}{0.48\textwidth}
        \centering
        \includegraphics[width=\textwidth]{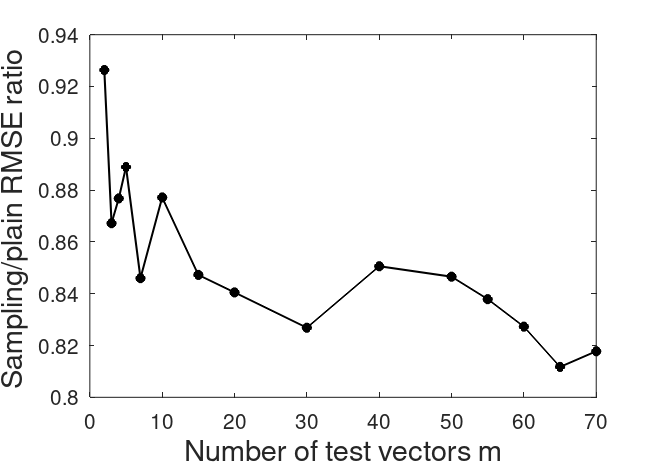}
        \caption*{(a) \texttt{exp}}
    \end{minipage}
    \caption{Resampling by orthogonal rotations $\mat{\Omega}\mat{U}$ reduces the variance in some cases, but not by much.}
    \label{fig:synthetic_ratios}
\end{figure}

\section{Conclusion and future work} \label{sec:concluding_remarks}
We have proposed two changes to \XTRACE{} that will reduce its variance without adding bias or requiring any additional matrix-vector products. In practice, the benefits are minor except for matrices with certain distributions of eigenvalues. 

The techniques outlined in \cref{sec:practical_implementation} could potentially be extended to cases where we deflate vectors over a larger Krylov space $[\mat{\Omega}, \mat{A}\mat{\Omega},\ldots,\mat{A}^{q-1}\mat{\Omega}]$. This can be useful in cases where we want to estimate the trace of $f(\mat{A})$ and matvecs with $f(\mat{A})$ are approximated by matrix polynomials, e.g., representing the number of triangles in a graph as $\tr(\mat{A}^3)$.

\begin{figure}
\begin{lstlisting}[language=Octave, caption=Matlab code for XTraceFull with resampling, label=listing:xtrace_code]
function est = xtraceFull(A, Omg, k)
  % A: the matrix whose trace is to be estimated
  % Omg: test vectors
  % k: number of different rotations of Omega to sample (default k=1)
  
  if nargin < 3
    k = 1;
  endif
  [N, m] = size(Omg);
  
  cnormc = @(M) M ./ vecnorm(M, 2, 1);
  diag_prod = @(A, B) sum(conj(A).*B, 1);
  
  [Q, R] = qr([Omg, A*Omg], 0);
  
  R(:,m+1:end) = R(:,m+1:end)/R(1:m,1:m);
  R(1:m,1:m) = eye(m);
  
  H = [R(:,m+1:end), Q'*(A*Q(:,m+1:end))];
  UU = eye(2*m);
  
  for j = 1:k
    S = cnormc(R'\UU);
    S(:,1:m) = S(:,1:m) - S(:,m+1:end) .* diag_prod(S(:,m+1:end), S(:,1:m));
    S(:,1:m) = cnormc(S(:,1:m));

    dSHS = diag_prod(S, H*S); % Vector of length m+m
  
    ests(j,:) = trace(H)*ones(1,m) - dSHS(m+1:end) + (N-2*m+1)*dSHS(1:m);
    
    UU = kron(eye(2), orth(randn(m)));
  endfor
  
  est = mean(ests(:));  
end
\end{lstlisting} 
\end{figure}

\bibliography{references}

@misc{girard1987algorithme,
  title={Un algorithme simple et rapide pour la validation crois{\'e}e g{\'e}n{\'e}ralis{\'e}e sur des probl{\`e}mes de grande taille},
  author={Girard, Didier},
  year={1987}
}

@inproceedings{Muirhead1982AspectsOM,
  title={Aspects of Multivariate Statistical Theory},
  author={Robb J. Muirhead},
  booktitle={Wiley Series in Probability and Statistics},
  year={1982},
  url={https://api.semanticscholar.org/CorpusID:123513635}
}

@article{epperly2024xtrace,
  title={X{T}race: Making the most of every sample in stochastic trace estimation},
  author={Epperly, Ethan N and Tropp, Joel A and Webber, Robert J},
  journal={SIAM Journal on Matrix Analysis and Applications},
  volume={45},
  number={1},
  pages={1--23},
  year={2024},
  publisher={SIAM}
}

@article{halmos1946theory,
  title={The theory of unbiased estimation},
  author={Halmos, Paul R},
  journal={The Annals of Mathematical Statistics},
  volume={17},
  number={1},
  pages={34--43},
  year={1946},
  publisher={Institute of Mathematical Statistics}
}

@inproceedings{bhattacharjee2025improved,
  title={Improved spectral density estimation via explicit and implicit deflation},
  author={Bhattacharjee, Rajarshi and Jayaram, Rajesh and Musco, Cameron and Musco, Christopher and Ray, Archan},
  booktitle={Proceedings of the 2025 Annual ACM-SIAM Symposium on Discrete Algorithms (SODA)},
  pages={2693--2754},
  year={2025},
  organization={SIAM}
}

@article{rhee2015unbiased,
  title={Unbiased estimation with square root convergence for {SDE} models},
  author={Rhee, Chang-{H}an and Glynn, Peter W},
  journal={Operations Research},
  volume={63},
  number={5},
  pages={1026--1043},
  year={2015},
  publisher={INFORMS}
}

@article{frommer2021analysis,
  title={Analysis of probing techniques for sparse approximation and trace estimation of decaying matrix functions},
  author={Frommer, Andreas and Schimmel, Claudia and Schweitzer, Marcel},
  journal={SIAM Journal on Matrix Analysis and Applications},
  volume={42},
  number={3},
  pages={1290--1318},
  year={2021},
  publisher={SIAM}
}

@article{hallman2022multilevel,
title = {A multilevel approach to stochastic trace estimation},
journal = {Linear Algebra and its Applications},
volume = {638},
pages = {125-149},
year = {2022},
issn = {0024-3795},
doi = {https://doi.org/10.1016/j.laa.2021.12.010},
url = {https://www.sciencedirect.com/science/article/pii/S0024379521004390},
author = {Eric Hallman and Devon Troester},
}

@article{frommer2022multigrid,
author = {Frommer, Andreas and Khalil, Mostafa Nasr and Ramirez-Hidalgo, Gustavo},
title = {A Multilevel Approach to Variance Reduction in the Stochastic Estimation of the Trace of a Matrix},
journal = {SIAM Journal on Scientific Computing},
volume = {44},
number = {4},
pages = {A2536-A2556},
year = {2022},
doi = {10.1137/21M1441894},
URL = { https://doi.org/10.1137/21M1441894
},
eprint = {  https://doi.org/10.1137/21M1441894}
}

@inproceedings{meyer2021hutch,
  title={Hutch++: Optimal Stochastic Trace Estimation},
  author={Meyer, Raphael A and Musco, Cameron and Musco, Christopher and Woodruff, David P},
  booktitle={Symposium on Simplicity in Algorithms (SOSA)},
  pages={142--155},
  year={2021},
  organization={SIAM}
}

@article{persson2022improved,
  doi = {10.1137/21m1447623},
  url = {https://doi.org/10.1137/21m1447623},
  year = {2022},
  month = jul,
  publisher = {Society for Industrial {\&} Applied Mathematics ({SIAM})},
  volume = {43},
  number = {3},
  pages = {1162--1185},
  author = {David Persson and Alice Cortinovis and Daniel Kressner},
  title = {Improved Variants of the {H}utch++ Algorithm for Trace Estimation},
  journal = {{SIAM} Journal on Matrix Analysis and Applications}
}
\bibliographystyle{abbrv}

\end{document}